\documentclass[12pt]{amsart}

\usepackage{amssymb,amsbsy,amsmath,amsfonts,amssymb,amscd}
\usepackage{latexsym}

\input xy
\xyoption{all}

\newcommand\Ga{\Gamma}
\newcommand\Lam{\Lambda}

\newcommand\ga{\gamma}

\DeclareMathOperator{\Fix}{Fix}
\DeclareMathOperator{\Ext}{Ext}
\DeclareMathOperator{\Hom}{Hom}
\DeclareMathOperator{\Alb}{Alb}

\newcommand{\CC}{\ensuremath{\mathbb{C}}}

\newcommand{\ZZ}{\ensuremath{\mathbb{Z}}}

\newcommand{\hol}{\ensuremath{\mathcal{O}}}

\newcommand{\PP}{\ensuremath{\mathbb{P}}}

\newcommand{\ra}{\ensuremath{\rightarrow}}

\def\eea{\end{eqnarray*}}
\def\bea{\begin{eqnarray*}}

\newcommand\dual{\mathrel{\raise3pt\hbox{$\underline{\mathrm{\thinspace d
\thinspace}}$}}}
\newcommand\qe{\ifhmode\unskip\nobreak\fi\quad $\Box$}       

\def\BOX{\hfill\lower.5\baselineskip\hbox{$\Box$}}

\newtheorem{theo}{Theorem}[section]
\newtheorem{remarkk}[theo]{Remark}

\newtheorem{defin}[theo]{Definition}

\newtheorem{prop}[theo] {Proposition}
\newtheorem{cor}[theo]{Corollary}
\newtheorem{lemma}[theo]{Lemma}
\newtheorem{example}[theo]{Example}

\newtheorem{conj}[theo]{Conjecture}

\newtheorem{rema}{Remark}[section]

\DeclareMathOperator{\Aut}{Aut}

\begin{document}

\title{The moduli space of Keum - Naie - surfaces}
\author{I. Bauer, F. Catanese}

\thanks{The present work took place in the realm of the DFG Forschergruppe 790
"Classification of algebraic surfaces and compact complex manifolds".}

\date{\today}

\maketitle
{\bf Alles Gute zum  60. Geburtstag, Fritz! \footnote{Our wishes did not come through, and we regret
very much the loss of an excellent mathematician, an exceptional person and a dear friend.}}
\section*{Introduction}

In the nineties Y. H. Keum and D. Naie (cf. \cite{naie94}, 
\cite{keum}) constructed a family of
minimal surfaces of general type with $K_S^2 =4$ and $p_g = 0$ as 
double covers of an Enriques surface with eight nodes.

They calculated the fundamental group of the constructed surfaces, 
but they did not address the problem of determining the moduli space 
of their surfaces.

\noindent
The motivation for the present paper comes from our joint work 
\cite{4names}  together with F. Grunewald and R. Pignatelli. In that 
article,  among other results, we constructed several series of new 
surfaces
of general type with $p_g=0$ as minimal resolutions of quotients of a 
product of two curves (of respective genera $g_1, g_2$ at least two)
by the action of a finite group $G$.  This construction produced many 
 interesting examples of new fundamental groups (of  surfaces
of general type with $p_g=0$) but in general yields proper 
subfamilies and not full  irreducible components of the respective moduli 
spaces of surfaces of general type (see also \cite{burniat1}, \cite{burniat2}). 

Obviously, when two such families 
yield surfaces with non isomorphic fundamental groups, then clearly 
the two families lie in distinct connected components of   the moduli 
space. But what happens if the fundamental groups are isomorphic (and 
the value of $K^2_S$
is the same)?

In particular, two of the families we constructed in \cite{4names} 
corresponded to surfaces having the
same fundamental group as the Keum-Naie surfaces. \footnote{ Observe 
however that the correct description of the fundamental group is 
only to be found  in \cite{naie94}.}

We reproduce below an excerpt of the classification table (of 
quotients as above by a non free action of $G$, but with canonical 
singularities) in \cite{4names}.

\begin{table}[ht]
\label{surfaces}
\begin{tabular}{|c|c|c|c|c|c|c|c|}
\hline
$K^2$& $T_1$  &  $T_2$&$g_1$&$g_2$&G& dim & $\pi_1(S)$\\
\hline\hline

4&$2^2$,
$4^2$&$2^2$, $4^2$&3& 3&$\ZZ / 4 \ZZ
\times \ZZ / 2 \ZZ$&2&$\ZZ^4 \hookrightarrow \pi_1
\twoheadrightarrow (\ZZ /2 \ZZ)^2$                   \\
\hline
4&$2^5$&
$2^5$&3&
3&$(\ZZ /2 \ZZ)^3$&4&$\ZZ^4 \hookrightarrow \pi_1
\twoheadrightarrow (\ZZ / 2 \ZZ)^2$\\
\hline
\end{tabular}
\end{table}

This excerpt shows the $2$ families, of respective dimensions $2$ and 
$4$, which we constructed  as $\ZZ /4 \ZZ \times \ZZ / 2 \ZZ$, resp. 
$(\ZZ /2 \ZZ)^3$, - coverings of $\PP^1 \times \PP^1$ and  branched on a 
divisor of type $(4,4)$,  resp. $(5,5)$ which are union of
horizontal and vertical lines ($T_1, T_2$ stand for the type of branching on each line).

Once we found  out that their fundamental groups were isomorphic to 
the fundamental groups of the surfaces constructed by Keum and Naie, 
the most natural question was whether all these surfaces
would belong to a unique irreducible component of the moduli space.

A straightforward computation showed that our family of dimension $4$ 
was equal to the family constructed by Keum, and that both families 
were subfamilies of the family constructed by Naie. To be more 
precise, each surface of our family of $(\ZZ / 2 \ZZ)^3$ - coverings 
of $\PP^1 \times \PP^1$ has $4$ nodes. These nodes can be smoothened 
simultaneously thus obtaining  a $5$ - dimensional family of $(\ZZ / 2 \ZZ)^3$ - 
Galois coverings of $\PP^1 \times \PP^1$. The full six dimensional 
component is obtained then as the family of natural deformations
(see \cite{cime2}) of the family of such  Galois coverings.

A somewhat lengthy but  essentially standard computation in local 
deformation theory  showed that the six dimensional family of natural 
deformations of smooth $(\ZZ /2 \ZZ)^3$ -  Galois coverings of $\PP^1 
\times \PP^1$ is an irreducible component of the moduli space. We 
will not give the details of this calculation, since we get a 
stronger result by a different method.

  The following theorem is the main result of this article:
\begin{theo}\label{main}
  Let $S$ be a smooth complex projective surface which is 
homotopically equivalent to a Keum - Naie surface. Then $S$ is a Keum 
- Naie surface. The connected component of the Gieseker moduli space 
$\mathfrak M^{can}_{1, 4}$ corresponding to
  Keum - Naie surfaces is irreducible, normal, unirational  of dimension  6.
\end{theo}

Observe that for surfaces of general type we have two moduli spaces: 
one is the moduli space
$\mathfrak M^{min}_{\chi, K^2}$ for minimal models $S$ having 
$\chi(\hol_S) = \chi$, $K^2_S = K^2$, the other is the moduli space 
$\mathfrak M^{can}_{\chi, K^2}$ for canonical models $X$ having 
$\chi(\hol_X) = \chi$, $K^2_X = K^2$; the latter is called the Gieseker 
moduli space and is a quasi projective scheme by Gieseker's theorem 
(\cite{gieseker}). Moreover, there is a natural morphism 
$\mathfrak M^{min}_{\chi, K^2} \ra \mathfrak M^{can}_{\chi, K^2}$ which is a 
bijection. The local structure of $\mathfrak M^{can}_{\chi, K^2}$ as
complex analytic space is the quotient of the base of the  Kuranishi 
family by the action
of the finite group $ \Aut(S) = \Aut(X)$. 

Usually the 
structure as analytic space of $\mathfrak M^{min}_{\chi, K^2}$ tends to be more singular 
than the one of
$\mathfrak M^{can}_{\chi, K^2}$ (see e.g. \cite{enr}).

In order to achieve our main result, we resort first of all to a 
slightly different construction of Keum - Naie surfaces.

  We start with a $(\ZZ/2 \ZZ)^2$ - action on the product of two 
elliptic curves $E_1' \times E_2'$.

  This action has $16$ fixed points and the quotient is an $8$ - nodal 
Enriques surface. Instead of constructing $S$ as the double cover of 
the Enriques surface, we consider an \'etale $(\ZZ / 2 \ZZ)^2$ - 
covering $\hat{S}$ of $S$, whose existence is guaranteed  from the 
structure of the fundamental group of $S$. $\hat{S}$ is obtained as a 
double cover of $E_1' \times E_2'$ branched in a $(\ZZ/2 \ZZ)^2$ - 
invariant divisor of type $(4,4)$,
and $S$ is recovered as the quotient of $\hat{S}$ by the action of 
$(\ZZ / 2 \ZZ)^2$ on it.

The structure of this $(\ZZ / 2 \ZZ)^2$  action and the geometry of 
the  covering $\hat{S}$ of $S$ is essentially encoded in the 
fundamental group $\pi_1(S)$, which is described as an affine group 
$\Gamma \in \mathbb{A}(2,\CC)$. In particular, it follows that  the 
Albanese map of $\hat{S}$ is the above double cover $\hat{\alpha} : 
\hat{S} \rightarrow E_1' \times E_2'$.

If $S'$ is now homotopically equivalent to a Keum - Naie surface $S$, 
then we have a corresponding
\'etale $(\ZZ / 2 \ZZ)^2$ - covering $\hat{S'}$ which is 
homotopically equivalent to $\hat{S}$. Since we know that the degree 
of the Albanese map of $\hat{S}$ is equal to two (by construction), 
we can conclude the same for the Albanese map of $\hat{S'}$ and this 
allows to deduce that also $\hat{S'}$ is a double cover of a product 
of elliptic curves branched in a $(\ZZ/2 \ZZ)^2$ - invariant divisor 
of type $(4,4)$.

\smallskip
Our paper is organized as follows: in section one we study a certain 
$(\ZZ / 2 \ZZ)^2$ - action on a product of two elliptic curves $E_1' 
\times E_2'$ and explain our construction of Keum - Naie surfaces.

In section 2 we use elementary representation theory to calculate the 
dimension of the space of $(\ZZ / 2 \ZZ)^2$ - invariant divisors of 
type $(4,4)$ on $E_1' \times E_2'$, and show that the
Gieseker moduli space of Keum - Naie surfaces is a normal, 
irreducible, unirational variety of dimension six.

In section 3 we conclude the proof of our main result \ref{main}.

The brief section 4 is devoted to the bicanonical image of
Keum-Naie surfaces: we show that the map has degree 4 and
that the image is always the same 4-nodal Del Pezzo surface of degree 4.

We stick to the traditional (`old fashioned'?) notation $\equiv$ to 
denote linear equivalence.

\section { A $(\ZZ/ 2 \ZZ)^2$ - action on a product of elliptic 
curves and Keum-Naie surfaces}

Let $(E,o)$ be any elliptic curve, with an action of the group
$$G : = (\ZZ /2 \ZZ)^2 = \{0, g_1, g_2, g_3 : = g_1+g_2 \}$$
given by
$$
g_1(z) := z + \eta, \ \ g_2(z) = -z,
$$
where $\eta \in E$ is a $2$ - torsion point of $E$.

\begin{rema}\label{invdiv}
The effective divisor $[o] + [\eta] \in Div^2(E)$ is invariant under 
$G$, hence the invertible
sheaf  $\hol_E([o] + [\eta])$ carries a natural $G$-linearization.

In particular, $G$ acts on the vector space $ H^0(E, \hol_E([o] + 
[\eta]))$ which splits then as a direct
sum $$H^0(E, \hol_E([o] + [\eta])) = \bigoplus_{\chi \in G^*} H^0(E, 
\hol_E([o] + [\eta]))^{\chi}$$
of the eigenspaces corresponding to the characters  $\chi$  of $G$. 
We shall use the self explanatory notation
$H^0(E, \hol_E([o] + [\eta]))^{+-}$ for the eigenspace corresponding 
to the character
$\chi$ such that $\chi(g_1) = 1$, $\chi(g_2) = -1$.
\end{rema}

We have the following:

\begin{lemma}\label{h++}
In the above setting we have $H^0(E, \hol_E([o] + [\eta]))^{+-} = 
H^0(E, \hol_E([o] + [\eta]))^{-+} =0$
and we have a splitting as a sum of two 1-dimensional eigenspaces:
  $$H^0(E, \hol_E([o] + [\eta])) = H^0(E, \hol_E([o] + [\eta]))^{++} 
\oplus H^0(E, \hol_E([o] + [\eta]))^{--}.$$

\end{lemma}

\begin{proof}
Obviously, since the $G$ linearization is obtained by considering the 
vector space of rational
functions with polar divisor at most $[o] + [\eta]$, the  subspace 
$H^0(E, \hol_E([o] + [\eta]))^{++} $
has dimension at least 1. On the other hand, there are exactly two 
$G$ invariant divisors
in the linear system $ | [o] + [\eta] |$.

Since, if $[P] + [Q] \in | [o] + [\eta] |$ is $G$ invariant, then 
$g_1([P] + [Q]) = [P+ \eta] + [Q+ \eta] = [P] + [Q]$, hence $[P+ 
\eta] = [Q]$ . Since $[P] + [Q] \equiv [o] + [\eta]$, $P, Q$ are $2$ 
- torsion points of $E$ (which automatically implies $g_2([P] + [Q]) 
= [-P] + [-Q] = [P] + [Q]$), and we have shown
that there are exactly two $G$-invariant divisors.

Therefore $H^0(E, \hol_E([o] + [\eta]))$ splits as the direct sum of 
two 1-dimensional eigenspaces,
one of which is $H^0(E, \hol_E([o] + [\eta]))^{++} $.

It suffices now to show that

$$H^0(E, \hol_E([o] + [\eta]))^{+-} = H^0(E, \hol_E([o] + [\eta]))^{-+} =0.$$

In fact, if this were not the case, all the divisors in the linear 
system $ | [o] + [\eta] |$ would be
invariant by either $g_1$ or by $g_2$.

  The first possibility was already excluded above,
while the second one means that, for each point $P$,   $[P] + [\eta - 
P] \in | [o] + [\eta] |$ satisfies $g_2([P] + [\eta - P]) = [-P] + [P 
-
\eta] = [P] + [\eta - P]$, which implies $ [P] = [-P]$, a contradiction.

\end{proof}

Consider  now two complex elliptic curves $E'_1, E'_2$, which can be 
written as quotients
$E_i' := \CC / \Lambda'_i$, $i=1,2$, with $\Lambda'_i := \ZZ e_i 
\oplus \ZZ e_i'$.

We consider the affine transformations $\gamma_1, \ \gamma_2 \in 
\mathbb{A}(2,\CC)$, defined as follows:
$$
\gamma_1 \begin{pmatrix}
   z_1\\z_2
\end{pmatrix} := \begin{pmatrix}
   z_1 + \frac{e_1}{2}\\- z_2
\end{pmatrix}, \ \ \gamma_2 \begin{pmatrix}
   z_1\\z_2
\end{pmatrix} := \begin{pmatrix}
   - z_1\\ z_2 + \frac{e_2}{2}
\end{pmatrix},
$$

\noindent
and let $\Gamma \leq \mathbb{A}(2,\CC)$ be the affine group generated 
by $\gamma_1, \gamma_2$ and by the translations $e_1, e_1', e_2, 
e_2'$.
\begin{rema}\label{enriqu}
i) $\Gamma$ contains the lattice $ \Lam'_1 \oplus \Lam'_2$,
hence $\Ga$ acts  on $E_1' \times E_2'$ inducing a faithful action
of $G:= (\ZZ /2 \ZZ)^2$ on $E_1' \times E_2'$.

\noindent
ii) While $\gamma_1, \ \gamma_2$ have no fixed points on $E_1' \times 
E_2'$, the involution $\gamma_1 \gamma_2$ has $16$ fixed points on 
$E_1' \times E_2'$. It is easy to see that the quotient $Y:=(E_1' 
\times E_2')/G$ is an Enriques surface having $8$ nodes, with 
canonical double cover the Kummer surface $(E_1' \times 
E_2')/<\gamma_1 \gamma_2>$.
\end{rema}

We will in the sequel lift the $G$ - action on $E_1' \times E_2'$ to 
an appropriate ramified double cover $\hat{S}$ and in such a way that 
$G$ acts freely on $\hat{S}$.

Consider the  geometric line bundle $\mathbb{L}$ on $E_1' 
\times E_2'$, whose invertible sheaf of sections is given by:
$$
\hol_{E_1'\times E_2'}(\mathbb{L}) :=p_1^*\hol_{E_1'}([o_1] + 
[\frac{e_1}{2}]) \otimes p_2^*\hol_{E_2'}([o_2] + [\frac{e_2}{2}]),
$$
where $p_i: E_1'\times E_2' \rightarrow E_i'$ is the projection onto 
the i-th factor.

\medskip

\begin{rema}\label{twist}
By remark \ref{invdiv}, the divisor $[o_i] + [\frac{e_i}{2}] \in 
Div^2(E_i')$ is invariant under $G$. Whence, we get a natural 
$G$-action on  $\mathbb{L}$.
But this is not the $G$-action on  $\mathbb{L}$ that we shall consider.

In fact, any two $G$ - actions on $\mathbb{L}$ differ by a character 
$\chi : G \rightarrow \CC^*$. We shall twist the above natural action 
of $\mathbb{L}$ by the character such that $\chi(\gamma_1) = 1$, 
$\chi(\gamma_2) = -1$. We shall call this twisted $G$-action   the 
canonical one.
\end{rema}

\begin{defin}
Consider the canonical $G$-action  on  $\mathbb{L}$ and on all its 
tensor powers,
and let
$$f \in H^0(E_1' \times E_2', p_1^*\hol_{E_1'}(2[o_1] + 
2[\frac{e_1}{2}]) \otimes p_2^*\hol_{E_2'}(2[o_2] + 
2[\frac{e_2}{2}]))^G$$
be a $G$ - invariant section of $\mathbb{L}^{\otimes 2}$.

  Denoting by $w$ a fibre coordinate of $\mathbb{L}$, let $\hat{X}$ be 
the double cover of $E_1' \times E_2'$ branched in $\{ f=0\}$, i.e., 
set $$
\hat{X} = \{w^2 = f(z_1,z_2) \} \subset \mathbb{L}.
$$
Then $\hat{X}$ is a $G$ - invariant hypersurface in $\mathbb{L}$, and 
we define the canonical model of a {\em Keum-Naie surface} to be the 
quotient   of $\hat{X}$ by the  $G$ - action.

\noindent
More precisely, we define $S$ to be a {\em Keum - Naie surface}, if
\begin{itemize}
  \item $G$ acts freely on $\hat{X}$, and
\item $\{f = 0 \}$ has only {\em non-essential singularities}, i.e., 
$\hat{X}$ has canonical singularities
(at most rational double points);
\item $S$ is the minimal resolution of singularities of $X:= \hat{X} / G$.
\end{itemize}
\end{defin}

\begin{rema}
One might also call the above surfaces`primary Keum-Naie surfaces'. 
In fact a similar construction, applied to the case where the action 
of $G$ has fixed points at some nodal singularities of
some special $\hat{X}$, produces other surfaces, which could 
appropriately be named  `secondary Keum-Naie surfaces'.
\end{rema}
\begin{lemma}
If
$$
f \in H^0(E_1' \times E_2', p_1^*\hol_{E_1'}(2[o_1] + 
2[\frac{e_1}{2}]) \otimes p_2^*\hol_{E_2'}(2[o_2] + 
2[\frac{e_2}{2}]))^G
$$ is such that $\{(z_1,z_2) \in E_1' \times E_2' \ | \ f(z_1, z_2) = 
0 \} \cap \Fix(\gamma_1  \gamma_2) = \emptyset$, then $G$ acts freely 
on $\hat{X}$.
\end{lemma}

\begin{proof}
  Recall that $\gamma_1$, $\gamma_2$ do not have fixed points on $E_1' 
\times E_2'$, whence they have no fixed points on $\hat{X}$. Since by 
\ref{twist} $(\gamma_1  \gamma_2)(w) = -w$, it follows that $G$ acts 
freely on $\hat{X}$ if and only if $\{f=0\}$ does not intersect the 
fixed points of $\gamma_1  \gamma_2$ on $E_1' \times E_2'$.

\end{proof}

\begin{prop}
Let $S$ be a Keum - Naie surface. Then $S$ is a minimal surface of 
general type with
\begin{itemize}
  \item[i)] $K_S^2 = 4$,
\item[ii)] $p_g(S) = q(S) = 0$,
\item[iii)] $\pi_1(S) = \Gamma$.
\end{itemize}
\end{prop}

\begin{proof}
i) Let $\pi : \hat{X} \rightarrow E_1' \times E_2'$ be the above 
double cover branched on $\{f = 0 \}$. Then $K_{\hat{X}} \equiv 
\pi^*(K_{E_1' \times E_2'} + p_1^*([o_1] + [\frac{e_1}{2}]) + 
p_2^*([o_2] + [\frac{e_2}{2}]))$, whence $K_{\hat{X}}^2 = 2 \cdot 
(p_1^*([o_1] + [\frac{e_1}{2}]) + p_2^*([o_2] + [\frac{e_2}{2}]))^2 = 
2 \cdot 8 = 16$. Therefore $K_S^2 = K_X^2 = \frac{K_{\hat{X}}^2}{|G|} 
= 4$.

ii) Let $\sigma : \hat{S} \rightarrow \hat{X}$ be the minimal 
resolution of singularities of $\hat{X}$. Then $S = \hat{S} /G$, and
$$
H^0(S, \Omega^1_S) = H^0(\hat{S}, \Omega^1_{\hat{S}})^G.
$$

Since $\pi \circ \sigma : \hat{S} \rightarrow E_1' \times E_2'$ has 
degree $2$, it is the Albanese map of $\hat{S}$, and we have  that 
$H^0(\hat{S}, \Omega^1_{\hat{S}}) = H^0(E_1' \times E_2', 
\Omega^1_{E_1' \times E_2'}) \cong \CC dz_1 \oplus \CC dz_2$. Hence
$$
  H^0(S, \Omega^1_S)=  H^0(\hat{S}, \Omega^1_{\hat{S}})^G = 0 ,
$$
i.e., $q(S) = 0$.

Observe that since $G$ acts freely
$$
H^0(\hat{X},\hol(K_{\hat{X}}))^G = H^0(X,\hol(K_X))=H^0(S,\Omega^2_S).
$$
Consider now the decomposition of
$$
V:= H^0(\hat{X},\hol(K_{\hat{X}})) = H^0(\hat{X},\hol(K_{\hat{X}}))^+ 
\oplus H^0(\hat{X},\hol(K_{\hat{X}}))^-
$$
  in invariant and antiinvariant part for the action of the involution 
$\sigma$ of the double cover $\pi : \hat{X} \rightarrow E_1' \times 
E_2'$ ($\sigma (z_1,z_2,w) = (z_1,z_2,-w) $).

  Note that
\begin{itemize}
\item[a)] $H^0(\hat{X},\hol(K_{\hat{X}}))^+ = H^0(E_1' \times E_2', 
\Omega^2_{E_1' \times E_2'}) = \CC (dz_1\wedge dz_2)$,
\item[b)] $H^0(\hat{X},\hol(K_{\hat{X}}))^- \cong H^0(E_1' \times 
E_2', \Omega^2_{E_1' \times E_2'}(\mathbb{L}))$.
\end{itemize}

In the uniformizing coordinates the first summand a) is generated by 
$dz_1 \wedge dz_2$, which is an eigenvector for the $G$-action, with 
character $\chi$ such that $\chi(\ga_1) =\chi(\ga_2) =-1$. We shall 
call this eigenspace  $V^{--}$.

Each vector  $y $ in the  addendum  b) can be written as
$$
y= \frac{\varphi_1(z_1)\varphi_2(z_2)}{w} dz_1 \wedge dz_2,
$$
where $\varphi_i \in H^0(E_i', \hol_{E_i'}([o_i]+[\frac{e_i}{2}]))$.

Recall that (cf. lemma \ref{h++}) $H^0(E_i', 
\hol_{E_i'}([o_i]+[\frac{e_i}{2}])) =: H_i$ splits as $H_i^{++} 
\oplus H_i^{--}$ (observe that exchanging the roles of $g_1$ and 
$g_2$ in lemma
  \ref{h++} makes fortunately no difference).

  Using that $\gamma_1(w) = w$, $\gamma_2(w) =-w$ and that $dz_1 
\wedge dz_2 \in V^{--}$, we get:
\begin{multline}
\frac{\varphi_1(z_1)\varphi_2(z_2)}{w} dz_1 \wedge dz_2 \in V^{+-} \ 
\iff \varphi_1 \in H_1^{++} \ \wedge \ \varphi_2 \in H_2^{--} \ 
\rm{or} \\ \varphi_1 \in H_1^{--} \ \wedge \ \varphi_2 \in H_2^{++};
\end{multline}

\begin{multline}
\frac{\varphi_1(z_1)\varphi_2(z_2)}{w} dz_1 \wedge dz_2 \in V^{-+} \ 
\iff \varphi_1 \in H_1^{++} \ \wedge \ \varphi_2 \in H_2^{++} \ 
\rm{or} \\ \varphi_1 \in H_1^{--} \ \wedge \ \varphi_2 \in H_2^{--}.
\end{multline}
The above calculations show that both eigenspaces  $V^{-+} $, $V^{+-} 
$ are 2-dimensional.
Since the summand b) has dimension 4, we obtain then:
\begin{itemize}
  \item[i)] $H^0(\hat{X},\hol(K_{\hat{X}}))^{--} = \CC (dz_1\wedge dz_2)$,
\item[ii)] $H^0(\hat{X},\hol(K_{\hat{X}}))^{+-} = 
\{\frac{\varphi_1(z_1)\varphi_2(z_2)}{w} dz_1 \wedge dz_2 \ | \ 
(\varphi_1 \in H_1^{++} \ \rm{and} \ \varphi_2 \in H_2^{--}) \ 
\rm{or} \ (\varphi_1 \in H_1^{--} \ \rm{and} \ \varphi_2 \in 
H_2^{++})\}$ has dimension 2;
\item[iii)] $H^0(\hat{X},\hol(K_{\hat{X}}))^{-+} = 
\{\frac{\varphi_1(z_1)\varphi_2(z_2)}{w} dz_1 \wedge dz_2 \ | \ 
(\varphi_1 \in H_1^{++} \ \rm{and} \ \varphi_2 \in H_2^{++}) \ 
\rm{or} \ (\varphi_1 \in H_1^{--} \ \rm{and} \ \varphi_2 \in 
H_2^{--})\}$ has dimension 2;
\item[iv)] $H^0(\hat{X},\hol(K_{\hat{X}}))^{++} = 0$.
\end{itemize}
In particular, we get $p_g(S) = \dim H^0(\hat{X},\hol(K_{\hat{X}}))^{++} = 0$.

\noindent
iii) it suffices to show that the fundamental group of $ \hat{S}$ 
maps isomorphically to the fundamental group of $E'_1 \times E'_2$. 
By the theorem of Brieskorn-Tyurina (\cite{brieskorn1}, \cite{brieskorn2}, \cite{tjurina}) we can reduce to the 
case where
$ \hat{X}$ is smooth, since the parameter space is connected, and 
there is a non empty open set of smooth branch curves $D$.

When $D$ is smooth, we conclude by the Lefschetz type theorem of 
Mandelbaum and Moishezon (\cite{mm}, page 218),
since $D$ is ample.
\end{proof}

\section{The moduli space of Keum - Naie surfaces}

The aim of this section is to prove the following result

\begin{theo}\label{locmod}
The connected component of the Gieseker moduli space corresponding to 
Keum - Naie surfaces is normal, irreducible, unirational of dimension 
equal to 6. Moreover, the base of the Kuranishi family of the 
canonical model $X$ of a  Keum-Naie surface is smooth.
\end{theo}

In order to describe the moduli space of Keum-Naie surfaces we shall 
preliminarily describe
the vector space
$$
H^0(E_1' \times E_2', p_1^*\hol_{E_1'}(2[o_1] + 2[\frac{e_1}{2}]) 
\otimes p_2^*\hol_{E_2'}(2[o_2] + 2[\frac{e_2}{2}]))^G.
$$
We consider $E_1'$ (resp. $E_2'$) as a bidouble cover of $\PP^1$ 
ramified in $4$ points $\{0,1,\infty,P\}$ (resp. $\{0,1,\infty,Q\}$), 
where $G = (\ZZ /2 \ZZ)^2  = \{0,g_1,g_2,g_3 := g_1+g_2 \}$ acts as 
follows:
$$
g_1(z) = z+\frac{e_1}{2}, \ \ g_2(z) = -z \ \ \rm{on} \ \ E_1',
$$
$$
g_1(z) = -z, \ \ g_2(z) = z +\frac{e_2}{2}\ \ \rm{on} \ \ E_2'.
$$
We denote the respective bidouble covering maps from $E_i'$ to 
$\PP^1$ by $\pi_i$.
Observe moreover that the quotient of $E_1'$ by the action of $g_1$ 
is an elliptic curve
$E_1$, while the quotient of $E_1'$ by the action of $g_2$ (resp. 
$g_3$) is isomorphic
to $\PP^1$.

\begin{rema}
  It is immediate from the above remark that the character 
eigensheaves of the direct image sheaf ${\pi_1}_* \hol_{E'_1}$ for 
the bidouble cover $\pi_1 : E_1' \rightarrow \PP^1$ are:
$$
\mathcal{L}_1^{-+} = \hol_{\PP^1}(1), \ \mathcal{L}_1^{+-} 
=\hol_{\PP^1}(2), \ \mathcal{L}_1^{--} =\hol_{\PP^1}(1).
$$
In fact, for instance, the direct image on $\PP^1$ of the sheaf of functions
on $E'_1 / g_1$ must be $ \cong \hol_{\PP^1}\oplus \hol_{\PP^1}(-2)$
and it equals $  \hol_{\PP^1}\oplus ({\mathcal{L}_1^{+-}})^{-1}$.

Similarly for $\pi_2 : E_2' \rightarrow \PP^1$ we have the character sheaves
$$
\mathcal{L}_2^{-+} = \hol_{\PP^1}(2), \ \mathcal{L}_2^{+-} 
=\hol_{\PP^1}(1), \ \mathcal{L}_2^{--} =\hol_{\PP^1}(1).
$$
Since $\hol_{E_i'}(2[o_i] + 2[\frac{e_i}{2}]) = 
\pi_i^*(\hol_{\PP^1}(1))$ we get
$$
H^0(E_i',\hol_{E_i'}(2[o_i] + 2[\frac{e_i}{2}])) = H^0(\PP^1, 
\hol_{\PP^1}(1) \otimes (\pi_i)_* \hol_{E_i'}),
$$
  and therefore:
\begin{itemize}
\item[i)] $V_1^{++} := H^0(E_1',\hol_{E_1'}(2[o_1] + 
2[\frac{e_1}{2}]))^{++} = H^0(\PP^1, \hol_{\PP^1}(1)) \cong \CC^2$;
\item[ii)] $V_1^{+-} := H^0(\hol_{E_1'}(2[o_1] + 
2[\frac{e_1}{2}]))^{+-} = H^0(\hol_{\PP^1}(1) \otimes 
(\mathcal{L}_1^{+-})^{-1}) = 0$;
\item[iii)] $V_1^{-+} := H^0(\hol_{E_1'}(2[o_1] + 
2[\frac{e_1}{2}]))^{-+} = H^0(\hol_{\PP^1}(1) \otimes 
(\mathcal{L}_1^{-+})^{-1}) \cong \CC$;
\item[iv)] $V_1^{--} := H^0(\hol_{E_1'}(2[o_1] + 
2[\frac{e_1}{2}]))^{--} = H^0(\hol_{\PP^1}(1) \otimes 
(\mathcal{L}_1^{--})^{-1}) \cong \CC$;
\item[v)] $V_2^{++} := H^0(E_2',\hol_{E_2'}(2[o_2] + 
2[\frac{e_2}{2}]))^{++} = H^0(\PP^1, \hol_{\PP^1}(1)) \cong \CC^2$;
\item[vi)] $V_2^{+-} := H^0(\hol_{E_2'}(2[o_2] + 
2[\frac{e_2}{2}]))^{+-} = H^0(\hol_{\PP^1}(1) \otimes 
(\mathcal{L}_2^{+-})^{-1}) \cong \CC$;
\item[vii)] $V_2^{-+} := H^0(\hol_{E_2'}(2[o_2] + 
2[\frac{e_2}{2}]))^{-+} = H^0(\hol_{\PP^1}(1) \otimes 
(\mathcal{L}_2^{-+})^{-1}) = 0$;
\item[viii)] $V_2^{--} := H^0(\hol_{E_2'}(2[o_2] + 
2[\frac{e_2}{2}]))^{--} = H^0(\hol_{\PP^1}(1) \otimes 
(\mathcal{L}_2^{--})^{-1}) \cong \CC$.

\end{itemize}
\end{rema}

As a consequence of the above remark, we get
\begin{lemma}\label{44inv}
\begin{multline*}
1) \ \ H^0(E_1' \times E_2', p_1^*\hol_{E_1'}(2[o_1] + 
2[\frac{e_1}{2}]) \otimes p_2^*\hol_{E_2'}(2[o_2] + 
2[\frac{e_2}{2}]))^{++} = \\
= (V_1^{++} \otimes V_2^{++}) \oplus (V_1^{--} \otimes V_2^{--}) \cong \CC^5;
\end{multline*}
\begin{multline*}
2) \ \ H^0(E_1' \times E_2', p_1^*\hol_{E_1'}(2[o_1] + 
2[\frac{e_1}{2}]) \otimes p_2^*\hol_{E_2'}(2[o_2] + 
2[\frac{e_2}{2}]))^{--} = \\
= (V_1^{++} \otimes V_2^{--}) \oplus (V_1^{--} \otimes V_2^{++}) 
\oplus (V_1^{-+} \otimes V_2^{+-}) \cong \CC^5;
\end{multline*}
\end{lemma}

\begin{proof}
  This follows immediately from the above remark since
\begin{multline*}
1) \ \ H^0(E_1' \times E_2', p_1^*\hol_{E_1'}(2[o_1] + 
2[\frac{e_1}{2}]) \otimes p_2^*\hol_{E_2'}(2[o_2] + 
2[\frac{e_2}{2}]))^G = \\
= \bigoplus_{\chi \in G^*} (H^0(\hol_{E_1'}(2[o_1] + 2[\frac{e_1}{2}]))^{\chi} 
\otimes H^0(E_2',\hol_{E_2'}(2[o_2] + 2[\frac{e_2}{2}]))^{\chi^{-1}}) 
= \\
= (V_1^{++} \otimes V_2^{++}) \oplus (V_1^{--} \otimes V_2^{--}) 
\cong \CC^4 \oplus \CC;
\end{multline*}
and
\begin{multline*}
2) \ \ H^0(E_1' \times E_2', p_1^*\hol_{E_1'}(2[o_1] + 
2[\frac{e_1}{2}]) \otimes p_2^*\hol_{E_2'}(2[o_2] + 
2[\frac{e_2}{2}]))^{--}  = \\
= \bigoplus_{\chi \in G^*} (H^0(\hol_{E_1'}(2[o_1] + 
2[\frac{e_1}{2}]))^{\chi} \otimes H^0(E_2',\hol_{E_2'}(2[o_2] + 
2[\frac{e_2}{2}]))^{\chi^{-1} \chi'}) = \\
= (V_1^{++} \otimes V_2^{--}) \oplus (V_1^{--} \otimes V_2^{++}) 
\oplus (V_1^{-+} \otimes V_2^{+-}) \cong \CC^2 \oplus \CC^2 \oplus 
\CC,
\end{multline*}
where $\chi ' (g_1) = -1$, $\chi '(g_2) = -1$.

\end{proof}

Now we can conclude the proof of theorem \ref{locmod}.
\begin{proof}{(of thm. \ref{locmod})}
Note that $V_i^{++}$ is without base points, whence also $V_1^{++} 
\otimes V_2^{++}$ has no base points. Therefore a generic
$$f \in H^0(E_1' \times E_2', p_1^*\hol_{E_1'}(2[o_1] + 
2[\frac{e_1}{2}]) \otimes p_2^*\hol_{E_2'}(2[o_2] + 
2[\frac{e_2}{2}]))^G$$
has smooth  and irreducible zero divisor $D$ (observe that $D$ is ample).

We obtain a six dimensional rational family parametrizing  all the 
Keum- Naie surfaces simply by varying
the two points $P, Q$ in $\PP^1 \setminus \{0,1,\infty \}$, and 
varying $f$ in an open set of the
  bundle of 4 dimensional projective spaces associated to the rank 
five vector bundle with fibre
$$(V_1^{++} \otimes V_2^{++}) \oplus (V_1^{--} \otimes V_2^{--}).$$

We obtain an irreducible unirational algebraic subset  of the moduli space which, by
the results of the forthcoming section, is indeed a connected 
component of the Gieseker moduli space (cf. theorem \ref{homotopy}). 
The dimension of this component is equal to 6, since
if two surfaces $S, S'$ are isomorphic, then this isomorphism lifts 
to a $G$-equivariant
isomorphism between $\hat{S}$ and $\hat{S'}$, and we get in 
particular an isomorphism of
the corresponding Albanese surfaces carrying one branch locus $D$ to 
the other $D'$.
It is now easy to see that, since we have normalized the line bundle 
$\mathbb L$, the
  morphism of the base of the rational family to the moduli space is 
quasi finite.

We  shall show that for each canonical model $X$  the base $\mathfrak 
B_X$ of the Kuranishi family of deformations of $X$ is smooth of 
dimension 6. For this it suffices to show that the dimension of the
Zariski tangent space to  $\mathfrak B_X$ is at most 6, since we already
saw that  $ \dim (\mathfrak B_X) \geq 6$.

In fact we could also  show that for each canonical model $X$ the 
above six dimensional family
induces a morphism $\psi$ of the  smooth rational
base whose Kodaira-Spencer map is an isomorphism, whence  $\psi$ yields
an isomorphism of the base with $\mathfrak B_X$.

Observe moreover that the assertion about the normality of this
component of the Gieseker moduli space follows right away from the fact
that the moduli space $\mathfrak M_{\chi, K^2}$ is locally analytically
isomorphic to the quotient of the base of the Kuranishi family
by the action of the finite group $\Aut(X)$. Indeed, a quotient of a normal
space is normal, and the local ring of a complex algebraic variety is
normal if its corresponding analytic algebra is normal.

\medskip
Let now $X = \hat{X} / G$ be the canonical model of a Keum - Naie
surface. Note that
$$
\Ext^1(\Omega^1_X, \hol_X) = \Ext^1(\Omega^1_{\hat{X}}, \hol_{\hat{X}})^{G}
$$

and that

$$
\mathfrak B_X = \mathfrak B_{\hat{X}} \cap  \Ext^1(\Omega^1_X, \hol_X) = \mathfrak B_{\hat{X}} \cap  \Ext^1(\Omega^1_{\hat{X}}, \hol_{\hat{X}})^{G}. 
$$

In order  to conclude the proof, it suffices therefore to show that 
$$ \mathfrak B_{\hat{X}} =  \Ext^1(\Omega^1_{\hat{X}}, \hol_{\hat{X}}),$$
which shows that $ \mathfrak B_{\hat{X}} $ is smooth.

We consider then $\hat{X}$ as a double cover of its Albanese variety $A$, and observe that the family of
such double covers of a principally polarized Abelian surface has dimension equal to $18= 3 + 15$, since Abelian surfaces 
depend on three moduli, and the branch divisor $D$ varies in a linear system of projective dimension
$\frac{1}{2}  D^2 - 1= 16 - 1 = 15$ (observe that changing the divisor class to an algebraically equivalent one 
can be achieved by a translation, which does not change the isomorphism class of the double cover).

Hence we are done once we show that $\dim \Ext^1_{\hol_{\hat{X}}}(\Omega^1_{\hat{X}},\hol_{\hat{X}}) = 18$.
This is the content of the following proposition, where we split all the relevant cohomology groups
in eigenspaces for the action of the group $\ZZ / 2 \ZZ$ generated by the covering involution for
the Albanese morphism.

\begin{prop}
\begin{enumerate}
 \item $\dim \Ext^1_{\hol_{\hat{X}}}(\Omega^1_{\hat{X}},\hol_{\hat{X}})^+ = 18$; 
\item $\dim \Ext^1_{\hol_{\hat{X}}}(\Omega^1_{\hat{X}},\hol_{\hat{X}})^- = 0$;
\item $\dim \Ext^2_{\hol_{\hat{X}}}(\Omega^1_{\hat{X}},\hol_{\hat{X}})^+ = 2$;
\item $\dim \Ext^2_{\hol_{\hat{X}}}(\Omega^1_{\hat{X}},\hol_{\hat{X}})^- = 8$.
\end{enumerate}

\end{prop}

\begin{proof}
Consider the following exact sequence:

\begin{multline}\label{ext1seq}
   0  \rightarrow
\Hom_{\hol_{A}}(\Omega^1_A,\hol_A) \rightarrow 
H^0(\hol_D(D))  
\rightarrow 
\Ext^1_{\hol_{\hat{X}}}(\Omega^1_{\hat{X}},\hol_{\hat{X}})^+ \rightarrow \\
\rightarrow 
\Ext^1_{\hol_A}(\Omega^1_A,\hol_A) \rightarrow 
H^1(\hol_D(D)) 
\rightarrow 
\Ext^2_{\hol_{\hat{X}}}(\Omega^1_{\hat{X}},\hol_{\hat{X}})^+ \rightarrow \\
\rightarrow
\Ext^2_{\hol_A}(\Omega^1_A,\hol_A)  \rightarrow  0, 
\end{multline}

 for  which a convenient reference is \cite{manettiinv}, where  the following is proven:
 
 \begin{prop}
   For every locally simple normal flat $(\ZZ/2 \ZZ)^r$ - cover $f: X
\rightarrow Y$
there is a $(\ZZ/2 \ZZ)^r$ - equivariant exact sequence of sheaves
\begin{equation}\label{equivsequ}
   0 \rightarrow f^* \Omega_Y^1 \rightarrow \Omega^1_X \rightarrow
\bigoplus_{\sigma \in (\ZZ/2 \ZZ)^r} \hol_{R_{\sigma}} (- R_{\sigma}) 
\rightarrow 0,
\end{equation} where $R_{\sigma}$ is the divisorial part of $\Fix(\sigma)$.

 Moreover,  for each $\sigma \in (\ZZ/2 \ZZ)^r$ and $i \geq 1$, we have
$$
\Ext^i_{\hol_X} (\hol_{R_{\sigma}} (- R_{\sigma}),\hol_X) \cong
\bigoplus_{\{\chi |
\chi(\sigma) = 0\}} H^{i-1}(\hol_{D_{\sigma}} (D_{\sigma} - \mathcal
L_{\chi})).
$$
\end{prop}
Observe that

\begin{itemize}
\item $\Hom_{\hol_{A}}(\Omega^1_A,\hol_A) \cong \CC^2$,
\item $\Ext^1_{\hol_{A}}(\Omega^1_A,\hol_A) \cong \CC^4$,
\item $\Ext^2_{\hol_{A}}(\Omega^1_A,\hol_A) \cong \CC^2$,
\item $H^1(\hol_D(D)) (= H^1(\hol_D(K_D))) \cong \CC$,
\item $H^0(\hol_D(D))(= H^0(\hol_D(K_D))) \cong \CC^{17}$, since $D$ has genus $ g=17$ (in fact $2(g-1) = D^2 = 32$).
\end{itemize}

Note that the map $ \lambda : \Ext^1_{\hol_A}(\Omega^1_A,\hol_A) \rightarrow 
H^1(\hol_D(D))$ is the Serre dual of 
$$
\CC \cong H^0(D,\hol_D) \rightarrow H^1(A, \Omega^1_A), \ \ 1 \mapsto c_1(D),
$$
which is injective. Therefore $\lambda$ is surjective, and part 1) and 2) of the claim follow.

In order to calculate the antiinvariant parts of $\Ext^i_{\hol_{\hat{X}}}(\Omega^1_{\hat{X}},\hol_{\hat{X}})$, $i=1,2$, observe that 
$$
\Ext^i_{\hol_{\hat{X}}}(\Omega^1_{\hat{X}},\hol_{\hat{X}})^- \cong \Ext^i_{\hol_{A}}(\Omega^1_A,\hol_A (-L)).
$$

But $\Ext^i_{\hol_{A}}(\Omega^1_A,\hol_A (-L)) \cong H^i(A, \Theta_A(-L)) \cong H^i(A, \hol_A(-L))^{\oplus 2}
\cong (H^{2-i}(A, \hol_A(L))^{\oplus 2})^{\vee}$.
\end{proof}

\end{proof}

\section{The fundamental group of Keum - Naie surfaces}

In the previous sections we proved that Keum - Naie surfaces form
a normal unirational irreducible component of dimension 6 of the
Gieseker moduli space. In this section we shall prove that indeed they form a
connected component. More generally, we shall prove the following:

\begin{theo}\label{homotopy}
  Let $S$ be a smooth complex projective surface which
is homotopically equivalent to a Keum - Naie surface. Then $S$ is a Keum - Naie
surface.

\end{theo}

\noindent
Let $S$ be a smooth complex projective surface with $\pi_1(S) = \Gamma$
($\Ga$ being the fundamental group of a Keum-Naie surface).

Recall that $\gamma_i^2 = e_i$ for $i = 1,2$. Therefore $\Gamma = \langle
\gamma_1, e_1', \gamma_2, e_2' \rangle$ and recall that, as we
observed in section 1, we
have the exact sequence
$$
1 \rightarrow \ZZ^4 \cong \langle e_1, e_1', e_2, e_2' \rangle 
\rightarrow \Gamma \rightarrow (\ZZ / 2 \ZZ)^2 \rightarrow 1,
$$

where $\ga_1 \mapsto (1,0) $, $\ga_2 \mapsto (0,1) $.

We have set $\Lambda_i':= \ZZ e_i \oplus \ZZ e_i'$, so that
  $\pi_1 (E_1' \times E_2')
=
\Lambda_1' \oplus \Lambda_2'$. 

We define also the two lattices 
$\Lambda_i := \ZZ
\frac{e_i}{2} \oplus \ZZ e_i'$.

\begin{rema}
  1) $\Gamma$ acts as a group of affine transformations on
the lattice $\Lambda_1 \oplus
\Lambda_2$.

\noindent
2) We have an \'etale double cover
$E_i' = \CC / \Lambda_i' \rightarrow E_i := \CC / \Lambda_i$, which 
is the quotient
by a semiperiod of $E_i'$, namely $e_i/2$.
\end{rema}

$\Gamma$ has two subgroups of index two:
$$
\Gamma_1 := \langle \gamma_1, e_1', e_2, e_2' \rangle, \ \ 
\Gamma_2:=\langle e_1, e_1', \gamma_2, e_2' \rangle,
$$
corresponding to two \'etale covers of $S$: $S_i \rightarrow S$, for $i = 1,2$.

\begin{lemma}\label{albsi}
The Albanese variety of $S_i$ is $E_i$. In particular, $q(S_1) = q(S_2) = 1$.
\end{lemma}

\begin{proof}
  Denoting the translation by $e_i$ by $t_{e_i} \in \mathbb{A}(2, 
\CC)$ we see that
$$
\gamma_1 t_{e_2} = t_{e_2}^{-1} \gamma_1 ,\ \ \gamma_1 t_{e_2'} = t_{e_2'}^{-1}
\gamma_1, \ \ \gamma_1 t_{e_1'} = t_{e_1'} \gamma_1.
$$

This implies that $t_{e_2}^2, t_{e_2'}^2 \in [\Gamma_1, \Gamma_1]$,
  and we get a surjective homomorphism
$$
\Gamma_1' := \Gamma_1 / 2 \langle e_2, e_2' \rangle \cong \Gamma_1 / 2 \ZZ^2
\rightarrow \Gamma_1^{ab} = \Gamma_1 / [\Gamma_1, \Gamma_1].
$$
Since $\gamma_1$ and $e_1'$ commute, we have that $\Ga'_1$ is commutative,
hence
$$
\Gamma_1' \cong \langle \gamma_1, e_1' \rangle \oplus (\ZZ / 2 \ZZ )^2 \cong
  \ZZ
\frac{e_1}{2} \oplus \ZZ e_1'\oplus (\ZZ / 2 \ZZ )^2 = \Lambda_1 
\oplus (\ZZ / 2
\ZZ )^2.
$$
Since $\Gamma_1'$ is abelian  $\Gamma_1' = \Gamma_1^{ab} = H_1(S_1,
\ZZ )$. This implies that $\Alb(S_1) = \CC / \Lambda_1 = E_1$.

\noindent
The same calculation shows that $\Gamma_2^{ab} = H_1(S_2, \ZZ )
  = \Lambda_2 \oplus (\ZZ / 2 \ZZ)^2$, whence $\Alb(S_2) = \CC / \Lambda_2 =
E_2$.

\end{proof}

For the sake of completeness we prove the following 
\begin{lemma}
$H_1(S, \ZZ) = \Gamma^{ab} = \ZZ / 
4 \ZZ \oplus (\ZZ/2 \ZZ)^3$.
\end{lemma}
\begin{proof}
 We have already seen in the proof of lemma \ref{albsi} that
$$
\gamma_1 t_{e_2} = t_{e_2}^{-1} \gamma_1 ,\ \ \gamma_1 t_{e_2'} = t_{e_2'}^{-1}
\gamma_1;
$$
$$
\gamma_2 t_{e_1} = t_{e_1}^{-1} \gamma_2 ,\ \ \gamma_2 t_{e_1'} = t_{e_1'}^{-1}
\gamma_2,
$$
and moreover, for $i=1,2$, we have that $\gamma_i$ commutes with $e_i, e_i'$.

This shows that we have a surjective homomorphism
$$
\Gamma' := \Gamma / \langle 2e_1, 2e_1', 2e_2,2e_2' \rangle \cong \Gamma/ 2 \ZZ^2 \ra \Gamma / [ \Gamma, \Gamma].
$$
Since $\gamma_2 \gamma_1 = t_{e_2} t_{e_1}^{-1} \gamma_1 \gamma_2$, it follows that $e_2 - e_1 \in [\Gamma, \Gamma]$, whence we have a surjective homomorphism

$$
\Gamma'' := \Gamma' / \langle e_1 - e_2 \rangle  \ra \Gamma / [ \Gamma, \Gamma],
$$
and it is easy to see that the homomorphism $\psi: \Gamma'' \ra \ZZ / 4 \ZZ \oplus (\ZZ / 2 \ZZ)^3$, given by  
$$
\psi(\overline{\gamma}_1) = (1,0,0,0), \ \psi(\overline{\gamma}_2) = (1,1,0,0), 
$$
$$ \psi(\overline{e'}_1) = (0,0,1,0), \ \psi(\overline{e'}_2) = (0,0,0,1).
$$
is well defined and is an isomorphism. This shows the claim.

\end{proof}

Let $\hat{S} \rightarrow S$ be the \'etale $(\ZZ / 2 \ZZ)^2$ - 
covering associated to $\Lambda_1'
\oplus \Lambda_2'
  = \langle e_1, e_1', e_2, e_2' \rangle \triangleleft \Gamma$. Since
$\hat{S} \rightarrow S_i \rightarrow S$, and $S_i$ maps to $E_i$ (via 
the Albanese map), we get a
morphism
$$
f:\hat{S} \rightarrow E_1 \times E_2 = \CC/ \Lambda_1 \times \CC / \Lambda_2.
$$
Then $f$ factors through the Albanese map of $\hat{S}$: but, since 
the fundamental group of
$\hat{S}$ equals  $\Lambda_1'
\oplus
\Lambda_2' $, and the covering of
$E_1
\times E_2$ associated to
$\Lambda_1'
\oplus
\Lambda_2' \leq
\Lambda_1 \oplus \Lambda_2$ is $E_1' \times E_2'$,  we see that $f$ 
factors through $E_1' \times
E_2'$ and that the Albanese map of $\hat{S}$ is $\hat{\alpha} : 
\hat{S} \rightarrow E_1' \times E_2'$.

We will conclude the proof of theorem \ref{homotopy} with the following

\begin{prop}\label{dc}
Let $S$ be a smooth complex projective surface, which is 
homotopically equivalent to a
Keum - Naie surface. Let $\hat{S} \rightarrow S$ be the \'etale $(\ZZ 
/ 2 \ZZ)^2$ - cover associated to
$\langle e_1, e_1', e_2, e_2' \rangle \triangleleft \Gamma$ and let
\begin{equation*}\label{stein}
\xymatrix{
\hat{S} \ar[r]^{\hat{\alpha}}\ar[dr]&E_1' \times E_2'\\
&Y \ar[u]_{\varphi}\\
}
\end{equation*}
be the Stein factorization of the Albanese map of $\hat{S}$.

\noindent
Then $\varphi$ has degree $2$ and $Y$ is a canonical model of $\hat{S}$.
\end{prop}

\begin{cor}
  $Y$ is a finite double cover of $E_1' \times E_2'$ branched on a 
divisor of type $(4,4)$.
\end{cor}

This completes the proof of theorem \ref{homotopy}.
\begin{proof}{of prop. \ref{dc}.} Consider the Albanese map
$\hat{\alpha} : \hat{S} \rightarrow E_1' \times E_2'$. Then we calculate
the degree of the Albanese map as the index of a certain subgroup of $H^4(\hat{S}, \ZZ)$,
namely:
$$
\deg (\hat{\alpha}) = [H^4(\hat{S}, \ZZ) : \hat{\alpha}^*H^4(E_1' 
\times E_2', \ZZ)
= \wedge ^4
\hat{\alpha}^*H^1(E_1' \times E_2', \ZZ)] =$$
$$  = [H^4(\hat{S}, \ZZ) : \wedge ^4 H^1(\hat{S}, \ZZ)].
$$
  But, since $S$ is homotopically equivalent to a Keum - Naie
surface $S'$, also $\hat{S}$ is homotopically equivalent to the 
\'etale $(\ZZ / 2 \ZZ)^2$ - covering
$\hat{S}'$ of $S'$. Since the 
  $[H^4(\hat{S}, \ZZ) : \wedge ^4 H^1(\hat{S}, \ZZ)]$ is a homotopy invariant,
and the degree of the Albanese map of $\hat{S}'$ is two, it follows 
that $\deg (\hat{\alpha}) = 2$.

\noindent
It remains to show that $Y$ has only rational double points. This 
follows from the following lemma.

\end{proof}

\begin{lemma}
  Let $A$ be an abelian surface and let $\hat{S}$ be a surface with 
$K_{\hat{S}}^2 = 16$
and $\chi(\hat{S}) = 4$. Moreover, let $\varphi : \hat{S} \rightarrow 
A$ be a generically finite morphism
of degree $2$. Then the branch divisor of $\varphi$ has only non 
essential singularities (i.e., the local
multiplicities of the singular points are $ \leq 3$, and for each 
infinitely near point we have multiplicity
at most two, cf. \cite{hor});  equivalently, if
\begin{equation*}
\xymatrix{
\hat{S} \ar[r]^{\varphi}\ar[dr]&A\\
&Y \ar[u]_{\delta}\\
}
\end{equation*}
is the Stein factorization, then $Y$ has at most rational double 
points as singularities.

\end{lemma}

\begin{proof}
  We use the notation and results on double covers due to E. Horikawa 
(cf. \cite{hor}).
Consider the following diagram:
\begin{equation*}
\xymatrix{
\hat{S} \ar[r]^{\varphi}&A\\
S^* \ar[r] \ar[u]_{\sigma}&\tilde{A} \ar[u]_{\delta},\\
}
\end{equation*}
where $S^* \rightarrow A$ is the so-called {\em canonical
resolution} in the terminology of Horikawa.

This means
that $\tilde{A} \rightarrow A$ is a minimal sequence
of blow ups such that the reduced transform of the branch divisor of
$\varphi$ is smooth, so that
$S^*
\rightarrow \tilde{A}$ is a finite double cover with $S^*$ smooth, and
$S^* \rightarrow \hat{S}$ is a
sequence of blow ups of smooth points.

Then we have the following formulae:

\begin{equation}
  K_{S^*}^2 = K_{\hat{S}}^2 - t = 2(K_A + \mathcal{L})^2 -2\sum 
([\frac{m_i}{2}] -1)^2,
\end{equation}

\begin{equation}
  \chi(S^*) = \chi(\hat{S}) = \frac{1}{2}\mathcal{L}(K_A + 
\mathcal{L}) - \frac {1}
{2}\sum [\frac{m_i}{2}]([\frac{m_i}{2}] -1),
\end{equation}
where $t$ is the number of points on $\hat{S}$ blown up by $\sigma$, 
$\hol_A(2 \mathcal{L})
  \cong \hol_A(B)$, where $B$ is the branch divisor of the (singular) 
double cover $Y \rightarrow A$.
Finally $m_i \geq 2$ is the multiplicity of the branch curve in the 
$i$-th center of the successive blow up
of $A$. For details we refer to \cite{hor}.

Notice that $Y$ has R.D.P.s if and only if $\xi_i := [\frac{m_i}{2}] 
= 1$ for each singular point
(and for all infinitely near points).

In our situation, the above two equations read:
$$
K_{S^*}^2 = 16 -t = 2\mathcal{L}^2 - 2 \sum (\xi_i - 1)^2;
$$
$$
\chi(\hat{S}) = 4 = \frac{1}{2} \mathcal{L}^2 -\frac {1}{2}\sum 
\xi_i(\xi_i -1).
$$
This implies that

$$
  2\mathcal{L}^2- 2 \sum (\xi_i - 1)^2 +t = 16 = 2 \mathcal{L}^2 
-2\sum \xi_i(\xi_i -1),
$$

or, equivalently,
$$  t = - 2 \sum (\xi_i -1).$$

Since $ \xi_i \geq 1$ this is only possible iff $\xi_i = 1$ for 
all $i$ and $t=0$.

\end{proof}

\begin{rema}
  Note that the above equations also imply that in the case $A = E_1' 
\times E_2'$,
$\mathcal{L}$ has to be of type $(2,2)$ or $(1,4)$ (resp. $(4,1)$). But
a divisor of type $(1,4)$ cannot be
  $(\ZZ / 2
\ZZ)^2$  invariant. This proves the above corollary.
\end{rema}

In fact, we conjecture the following to hold true:
\begin{conj}\label{conj1}
  Let $S$ be a minimal smooth projective surface such that
\begin{itemize}
  \item[i)] $K_S^2 = 4$,
\item[ii)] $\pi_1(S) \cong \Gamma$.
\end{itemize}
Then $S$ is a Keum - Naie surface.
\end{conj}

In fact, we can prove
\begin{theo}\label{Kample}
  Let $S$ be a minimal smooth projective surface such that
\begin{itemize}
  \item[i)] $K_S^2 = 4$,
\item[ii)] $\pi_1(S) \cong \Gamma$,
\item[iii)] there is a deformation of $S$ having ample canonical bundle.
\end{itemize}
Then $S$ is a Keum - Naie surface.
\end{theo}

Before proving the above theorem, we recall the following results:

\begin{theo}[Severi's conjecture, \cite{pardini}]\label{sevconj}
  Let $S$ be a minimal smooth projective surface of maximal Albanese dimension
(i.e., the image of the Albanese map is a surface): then $K_S^2 \geq 
4 \chi(S)$.
\end{theo}

M. Manetti proved Severi's inequality under the stronger assumption that $K_S$ is ample,
but he also gave a description of the limit case $K_S^2 = 4 \chi(S)$, 
which will be crucial for our result.

\begin{theo}[M. Manetti,\cite{manetti}]\label{SevMan}
  Let $S$ be a minimal smooth projective surface of maximal Albanese 
dimension with $K_S$ ample:
  then
$K_S^2 \geq 4 \chi(S)$, and equality holds if and only if $q(S) = 2$, 
and the Albanese map $\alpha : S
\rightarrow \Alb(S)$ is a finite double cover.
\end{theo}

\begin{proof}{(of \ref{Kample})}
  We know that there is an \'etale $(\ZZ / 2 \ZZ)^2$ - cover $\hat{S}$ 
of $S$ with Albanese
  map $\hat{\alpha} : \hat{S} \rightarrow E_1' \times E_2'$.
The Albanese map of $\hat{S}$ must be surjective, otherwise the 
Albanese image, by the universal
property of the Albanese map, would be a curve $C$ of genus 2.
But then we would have a surjection  $\pi_1 (\hat{S}) \ra \pi_1 (C)$, 
which is a contradiction
since  $\pi_1 (\hat{S})$ is abelian and $\pi_1 (C)$ is not abelian.

Note that $K_{\hat{S}}^2 = 4 K_S^2 = 16$.
By Severi's inequaltiy, it follows that $\chi(\hat{S}) \leq 4$, but 
since $1 \leq \chi(S) = \frac{1}{4}
\chi(\hat{S})$, we have $\chi(\hat{S}) = 4$. Since $S$ deforms to a 
surface with $K_S$ ample, we can
apply Manetti's result and obtain that $\hat{\alpha} : \hat{S} 
\rightarrow E_1' \times E_2'$ has degree
$2$, and we conclude as before.
\end{proof}

  It seems reasonable to conjecture (cf. \cite{manetti}) the 
following, which would obviously
imply our conjecture \ref{conj1}.
\begin{conj}
  Let $S$ be a minimal smooth projective surface of maximal Albanese dimension.
Then $K_S^2 = 4 \chi(S)$ if and only if $q(S) = 2$, and the Albanese 
map has degree $2$.
\end{conj}

\begin{rema}
1) In \cite{keum} the author proves that Bloch's conjecture holds, 
i.e., $A_0(S) = \ZZ$,
for the family of surfaces he constructs. Since Keum constructs only 
a $4$ - dimensional subfamily of the
connected component of the moduli space, this does not  imply that 
Bloch's conjecture holds for all
Keum - Naie surfaces. Nevertheless, exactly the same proof holds in 
the general case, thereby showing
that Bloch's conjecture holds true for all Keum- Naie surfaces.
\end{rema}
\section{The bicanonical map of Keum-Naie surfaces}

\noindent
 It is shown in \cite{naie94} that the bicanonical map of a Keum-Naie surface is base point free and has degree 4. Moreover, in \cite{mlp}, the authors show that the bicanonical image of a Keum - Naie surface is a rational surface, and the bicanonical morphism factors through the double cover $S \ra Y$, where $Y = (E_1' \times E_2') / (\ZZ / 2 \ZZ)^2$ is an $8$-nodal Enriques surface. More precisely they show the following (cf. \cite{mlp}, 5.2.): minimal surfaces $S$ of general type with $p_g = 0$ and $K^2 = 4$ having an involution $\sigma$ such that 
\begin{itemize}
 \item[i)] $S/ \sigma$ is birational to an Eriques surface and
\item[ii)] the bicanonical map is composed with $\sigma$
\end{itemize}
are precisely the Keum - Naie surfaces.

As a corollary of our very explicit description of Keum-Naie surfaces we prove the following

\begin{theo}
The bicanonical map of a Keum-Naie surface is a finite iterated double covering
of the 4 nodal Del Pezzo surface  $\Sigma \subset \PP^4$  of degree 4, the complete intersection of the following two quadric hypersurfaces in $\PP^4$: 
$$
Q_1 = \{ z_0z_3 - z_1z_2 = 0\}, 
$$
$$
Q_2= \{  z_4^2 - z_0z_3= 0\}.
$$
\end{theo}

\begin{proof}
Observe first of all that $H^0(2 K_S) \cong H^0(2 K_{\hat{S}})^{++} $.
The standard formulae for the bicanonical system of a double cover allow to decompose
$H^0(2 K_{\hat{S}}) $ as the direct sum of the invariant part $U$ and the anti-invariant part 
$U'$ for the covering involution of the Albanese map.

We have then $H^0(2 K_{\hat{S}}) = U \oplus U'$, where
$$
U : = \{ \Phi (z_1,z_2) \frac{(dz_1 \wedge dz_2)^{\otimes 2}}{w^2}  \} = \hat{\alpha}^* H^0 (\hol_{E'_1 \times E'_2} (D)),
$$  
and
$$
U' : = \{ \Psi (z_1,z_2) \frac{(dz_1 \wedge dz_2)^{\otimes 2}}{w}  \} =  \hat{\alpha}^*H^0 (\hol_{E'_1 \times E'_2} (L)).
$$  
Here $\Phi (z_1,z_2) =\Phi_1 (z_1)\Phi_2 (z_2) $ is a section of $\mathbb{L}^{\otimes 2} = \hol_{E'_1 \times E'_2}(D)$, whereas $\Psi (z_1,z_2) =\Psi_1 (z_1)\Psi_2 (z_2) $ is a section of $\mathbb{L} = \hol_{E'_1 \times E'_2}(L)$.

Since however $w$ is an eigenvector for $G$ with character of type $(-,+)$, $w^2$ is a $G$-invariant,
and $U^{++} = \hat{\alpha}^* H^0 (\hol_{E'_1 \times E'_2} (D))^{++}$, while
${U'}^{++} = \hat{\alpha}^* H^0 (\hol_{E'_1 \times E'_2} (L))^{-+}$.

By the formulae that we developed in lemma \ref{h++} the second space is equal to 0, while
the formulae developed in section 2 show that  $$U^{++} =  (V_1^{++} \otimes V_2^{++}) \oplus (V_1^{--} \otimes V_2^{--}) \cong \CC^4 \oplus \CC.$$

The first consequence of this calculation is that the composition of $\hat{S} \ra S$ with the bicanonical map of $S$ factors through the product $E'_1 \times E'_2$.

Moreover, these sections are invariant for the action of the group $G$, and further for the action
of the automorphism $$g'(z_1, z_2) : = (- z_1 +\frac{e_1}{2}, z_2)$$ (observe that $G$ and $g'$ are contained in $(\ZZ/2\ZZ)^2 \oplus (\ZZ/2\ZZ)^2$).

Whence the above composition factors through the $(\ZZ/2\ZZ)$ quotient $\Sigma$ of the Enriques surface
$(E'_1 \times E'_2)/G$ by the action of $g'$. 

$\Sigma$ is a double cover of $\PP^1 \times \PP^1$ ramified in the union of two vertical plus two horizontal lines.
The subspace $(V_1^{++} \otimes V_2^{++})$ is the pull back of the hyperplane series of the Segre embedding of
$\PP^1 \times \PP^1$ , thus we get a basis of sections $ z_0,z_1,  z_2, z_3$ 
satisfying $ z_0z_3 - z_1z_2 = 0$. 

 We can complete these to a basis of $H^0(2K_S)$ by
choosing $z_4$ such that $z_4^2 = z_0 z_3$. 

Since $ H^0 (\hol_{E'_1 \times E'_2} (D))^{++}$ is base point free, the bicanonical map is a morphism,
factoring through the double cover $ S \ra Y$ and the double cover $ Y \ra \Sigma$.

It is immediate to see that $(z_0, z_1,  z_2, z_3, z_4)$ yield an embedding of $\Sigma$.
We get a complete intersection of degree 4, hence a Del Pezzo surface of degree 4.
The four nodes, which correspond to the 4 points where the 4 lines of the branch locus meet,
are seen to be the 4 points 
$$ z_4= z_1=z_2= z_3=0,$$
$$ z_4= z_1=z_2=z_0=0,$$
$$ z_4= z_0=z_3= z_1=0,$$
$$ z_4= z_0=z_3= z_2=0.$$

\end{proof}


\bigskip
\noindent
{\bf Authors' Adresses:}

\noindent
I.Bauer, F. Catanese \\
Lehrstuhl Mathematik VIII\\
Mathematisches Institut der Universit\"at Bayreuth\\
NW II\\
Universit\"atsstr. 30\\
95447 Bayreuth

\begin{verbatim}
ingrid.bauer@uni-bayreuth.de, 
fabrizio.catanese@uni-bayreuth.de
\end{verbatim}
\end{document}